\documentclass[12pt]{amsart}
\usepackage{amsmath, amsfonts, amssymb}
\usepackage{enumitem,dsfont}
\usepackage{color}
\usepackage{graphicx}
\usepackage{comment}

\newcommand{\jb}[1]{\left\langle #1 \right\rangle}

\newtheorem{thm}{Theorem}
\newtheorem{lm}{Lemma}

\numberwithin{equation}{section}
\numberwithin{thm}{section}
\numberwithin{lm}{section}
\makeatletter
\@namedef{subjclassname@2020}{%
  \textup{2020} Mathematics Subject Classification}
\makeatother
\title [Nonlinear wave equation with  damping and potential]
{Critical exponent for nonlinear wave equations with damping and potential terms}

\author{Masakazu Kato}
\address{M. Kato \newline
Faculty of Science and Engineering,
Muroran Institute of Technology,
Muroran 050-8585, Japan}
\email{mkato@mmm.muroran-it.ac.jp}
\author{Hideo Kubo}
\address{H. Kubo \newline
Department of Mathematics,
Faculty of Science, Hokkaido University,
Sapporo 060-0810, Japan}
\email{kubo@math.sci.hokudai.ac.jp}
\thanks{The first author was
partially supported by Grant-in-Aid for Science Research (No. 19H01795 and No. 22H00097),
JSPS. The second author was
partially supported by Grant-in-Aid for Science Research (No. 16H06339, No. 22H00097 and No. 26220702),
JSPS}


\begin{document}
\begin{abstract}
The aim of this paper is to determine the critical exponent for the nonlinear wave equations with damping and potential terms of the scale invariant order, by assuming that these terms satisfy a special relation. We underline that our critical exponent is different from the one 
for related equations such as the nonlinear wave equation without lower order terms, only with a damping term, and only with a potential term. Moreover, we study the effect of the decaying order of initial data at spatial infinity. In fact, we prove that not only the lower order terms but also the order of the initial data affects the critical exponent, as well as the sharp upper and lower bounds of the maximal existence time of the solution.
\end{abstract}

\maketitle

\section{Introduction}
This paper is concerned with the Cauchy problem for the nonlinear wave equation 
with damping and potential terms:
\begin{align} \label{1.1}
\left\{
\begin{array}{l}
(\partial_t^2 + 2w(r) \partial_t -\Delta +V(r))U=  |U|^p
\quad \mbox{in} \ (0,T) \times {\mathbb R}^3,
\\
U(0,x)=\varepsilon f_0(x), \quad (\partial_t U)(0,x)=\varepsilon f_1(x)
\quad \mbox{for} \ x \in {\mathbb R}^3,
\end{array}
\right.
\end{align}
where $r=|x|$, $p>1$, $\varepsilon>0$, and $f_0$, $f_1$ are given functions
vanishing at spatial infinity, like
\begin{equation} \label{1.2}
f_0(x)={\mathcal{O}}(|x|^{-\kappa}), \quad f_1(x)={\mathcal{O}}(|x|^{-\kappa-1}) \ \
\mbox{as}\ \ |x| \to \infty.
\end{equation}
Here $\kappa$ is a positive constant.

The Cauchy problem for the wave equation with power-type nonlinearity has a long history.
The starting point is the study for the case where both damping and potential terms are absent.
In this case, the critical exponent has been determined for general space dimensions $n$ with $n\ge 2$.
The exponent is obtained as the positive root of 
\begin{align} \label{gSdef}
\gamma_S(p,n):=2+(n+1)p-(n-1)p^2=0
\end{align}
and denoted by $p_S(n)$. 
We call it Strauss exponent after the Strauss conjecture.
The conjecture says that if $p>p_S(n)$, then there exists uniquely a global solution to 
$$
 (\partial_t^2  -\Delta) U=  |U|^p
\quad \mbox{in} \ (0,T) \times {\mathbb R}^n$$
for sufficiently small initial data,
and that if $1<p\le p_S(n)$, then the solution blows up in finite time
even for the small initial data
 (see \cite{Str89}, \cite{Joh79}, \cite{G81}, \cite{GLS97}, \cite{DGK01}, 
\cite{YZ06}, \cite{KubOht05}, or references in \cite{G2005}).

We underline that if we add a damping term $\mu (1+|x|^{2})^{-\alpha/2}\partial_t U$
with $\mu>0$ and $0\leq \alpha \leq 1$ in the above equation, i.e.
\begin{align}\label{eq1}
(\partial_t^2 +\mu (1+|x|^{2})^{-\alpha/2}\partial_t-\Delta)U=|U|^p,
\quad (t,x) \in (0,\infty) \times \mathbb{R}^{n},
\end{align}
then the critical exponent drastically changes.
In fact, when $\alpha=0$, Todorova and Yordanov \cite{TY01} showed that the critical exponent
becomes the Fujita exponent $p_F(n)=1+2/n$.
This is because the solution behaves rather similar to that of 
the heat equation $\partial_t v -\Delta v=v^p$.
When $0<\alpha<1$, Ikehata, Todorova and Yordanov \cite{ITY09} obtained that the critical exponent
shifts to $p_F(n-\alpha)=1+2/(n-\alpha)$. 
For the scale invariant case $\alpha=1$, 
Li \cite{Li13} showed the non-existence result of the global solution 
for $1<p\le p_F(n-1)$ with $p_F(0)=\infty$. 
On the one hand, an interesting blow-up result was derived by Ikeda and Sobajima \cite{IkeSob17space}
for the case where $0\leq \mu < (n-1)^2/(n+1)$, $n/(n-1) < p \leq p_S(n+\mu)$ and $n\geq 3$.
Actually, since $p_S(n)>p_F(n-1)$ and $p_S(n)$ is monotonically decreasing to 1, 
we see that the critical exponent is of the Strauss type if $\mu$ is small.
This type of shift is an analogue to the results due to 
D'Abbicco, Lucente and Reissig \cite{DAbLucRei15},  Ikeda, Sobajima \cite{IkeSob17time}, Kato, Sakuraba \cite{KatSak18}, and Lai \cite{Lai18}
in which the Cauchy problem for the wave equation with damping term of time dependent coefficient:
\begin{align*} 
\left( \partial_t^2 + 2(1+t)^{-1} \partial_t -\Delta \right)U=  |U|^p
\quad \mbox{in} \ (0,T) \times {\mathbb R}^3
\end{align*}
was studied. Indeed, the problem admits a global solution for sufficiently small initial data 
if $p>p_S(5)$, and that the solution blows up in  finite time if $1<p\le p_S(5)$.
This means that the critical exponent $p_S(3)$ is shifted to $p_S(3+2)$, 
by virtue of the presence of the damping term.


On the other hand, the critical exponent does not change when we add only potential
term $V(x) U$ with a non-negative function $V$  as
$$
 (\partial_t^2  -\Delta +V(x)) U=  |U|^p
\quad \mbox{in} \ (0,T) \times {\mathbb R}^n.$$
Indeed, it was shown in \cite{GHK01} that there exists small global solutions if $n=3$, $p>p_S(3)$ and 
$V \in C_0^{\infty}(\mathbb{R}^3)$.
When $V={\mathcal{O}}(|x|^{-2-\delta})$ as $|x| \to \infty$
for some $\delta>0$, the blow-up result was obtained
by \cite{YZ06}, provided $n \geq 3$  and $1<p<p_S(n)$
(see also \cite{ST97}, \cite{K05} for potentials with small magnitude).

Therefore, it is natural to ask whether the critical exponent might change or not 
if both damping and potential terms are in presence as in \eqref{1.1}. 
In Georgiev, Kubo and Wakasa \cite{GKW19},
by assuming that $w(r)$ takes the form of $1/r$ for large values of $r$ and
that $w(r)$, $V(r)$ have the following relation:
\begin{align} \label{1.5}
V(r)=-w^\prime(r) + w(r)^2\quad \mbox{for} \ r>0,
\end{align}
the critical exponent was shown to be shifted from $p_S(3)$ to $p_S(5)$.
Namely, the effect of the damping term appears as the shift of the critical exponent.
But it is not still clear if the size of the coefficient of $w(r)$ comes into play or not. 
For this reason, we assume that $w(r)$ is a function in
$C([0,\infty)) \cap C^1(0,\infty)$ satisfying
\begin{align} \label{1.3}
2w(r)=\frac{\mu}{r} +\widetilde{w}(r),\ \
|\widetilde{w}(r)| \lesssim r^{-1-\delta} \quad \mbox{for} \ r \ge r_0
\end{align}
with some positive numbers $r_0$, $\mu \geq 0$ and $\delta >0$.
Then, our question reformulate as follows:
Does the critical exponent $p_S(3)$ shift to $p_S(3+\mu)$ for any $\mu\ge 0$? 

In this paper, we shall give an affirmative answer to this question.
To be more precise, one more issue concerning the decaying order
of the initial data.
In fact, in \cite{GKW19}, the initial data is assumed to vanish sufficiently fast
order at spatial infinity.
However, in view of the work of Asakura \cite{Asakura}, the self-similarity becomes
another important factor which determines the blow-up and global existence for
small initial data
(see also \cite{Kubota93}, \cite{AgemiTakam92}, \cite{Tsutaya94}, \cite{Kubo97}).
Namely, the global behavior would be different between the cases $\kappa \ge 2/(p-1)$ 
and $\kappa<2/(p-1)$, where $\kappa$ is the number from \eqref{1.2}.
As a matter of fact, we prove a blow-up result for $1<p\le p_S(3+\mu)$ in Theorem 2.1
and a global existence result for $p>p_S(3+\mu)$ and $\kappa \ge 2/(p-1)$ in Theorem 2.2.
Moreover, we obtain lower bounds of the lifespan of the solution when
$1<p\le p_S(3+\mu)$ or $\kappa < 2/(p-1)$ in Theorem 2.3.
Finally, we show the optimality of the lifespan estimates by deriving upper bounds of lifespan
for slowly decreasing initial data in Theorem 2.4. 
In conclusion, we find from these results that our equation \eqref{1.1} has different nature
from the wave equation only with damping term, like \eqref{eq1}.

In the next section, we formulate our problem under the assumption of the radial symmetry
and describe the statements mentioned in the above, precisely.

\section{Formulation of the problem and Results}

Since we are interested in spherically symmetric solutions
to the problem (\ref{1.1}), we set
$$
u(t,r)=r U(t,r\omega) \quad \mbox{with} \ r=|x|,~ \omega=x/|x|.
$$
Then, by the relation \eqref{1.5} we obtain
\begin{align} \label{1.6}
\left\{
\begin{array}{l}
(\partial_t -\partial_r+w(r)) (\partial_t +\partial_r+w(r)) 
u={|u|^p}/{r^{p-1}}
\quad \mbox{in} \ (0,T) \times (0,\infty),
\\
u(0,r)=\varepsilon \varphi(r),
\quad (\partial_t u)(0,r)=\varepsilon \psi(r)
\quad \mbox{for} \ r>0,\\
u(t,0)=0 \quad \mbox{for} \ t  \in (0,T),
\end{array}
\right.
\end{align}
where we put $\varphi(r)=r f_0(r)$ and $\psi(r)=r f_1(r)$.

In order to express the solution of \eqref{1.6},
we set $W(r)=\displaystyle \int_0^r w(\tau) d\tau$ for $r \ge 0$  and define
\begin{align} \label{eq.d1}
E_-(t,r,y)=
e^{-W(r)} e^{2W( 2^{-1} (y- t+r) )} e^{-W(y)} \quad
\mbox{for}\ t, r \ge 0, ~ y \ge t-r.
\end{align}
From the assumption \eqref{1.3} we can deduce 
\begin{align*} 
e^{W(r)} \sim \langle r \rangle^{\mu/2},  \ \ r >0.
\end{align*}
Then the definition \eqref{eq.d1} of $E_-$ implies
\begin{align} \label{eq.le2}
E_-(t,r,y) \sim  \frac{\langle  r-t +y \rangle^{\mu}}{\langle r \rangle^{\mu/2} \langle y \rangle^{\mu/2}}.
\end{align}

Following the argument in \cite{GKW19}, we see that the problem \eqref{1.6} can be written
in the integral form
\begin{align} \label{eq.io0}
u(t,r)=
\varepsilon u_L(t,r)+
\frac12 \iint_{\Delta_-(t,r)} E_-(t-\sigma,r,y) \frac{|u(\sigma, y)|^p}{y^{p-1} } dy d\sigma
\end{align}
for $t>0$, $r>0$, where we have set
$$
\Delta_-(t,r)
=\{(\sigma, y) \in (0,\infty) \times (0,\infty); \ |t-r| < \sigma+y < t+r, \
\sigma-y <t-r \}.
$$
Besides, $u_L$ is the free solution defined by
\begin{align} \label{eq.io00}
u_L(t,r)=
& \frac12 \int_{|t-r|}^{t+r} E_-(t,r,y) \left( \psi(y) + \varphi^\prime(y)+ w(y)
\varphi(y) \right) dy
\\ \nonumber
& + \chi(r-t) E_-(t,r,r-t) \varphi(r-t),
\end{align}
where $\chi(s)=1$ for $s \ge 0$,
and $\chi(s)=0$ for $s<0$.

First of all, we extend the blow-up result in \cite{GKW19} where $\mu=2$
is assumed as follows.
Its proof can be found in \cite{KK22}.

\begin{thm} \label{bw2}
Suppose that \eqref{1.3} holds.
Let 
$\varphi$, $\psi \in C([0,\infty))$ satisfy
\begin{align*} 
\varphi (r)\equiv 0, \quad \psi(r) \ge 0  \ (\not\equiv 0) \quad \mbox{for}\ r\ge 0.
\end{align*}
If $1<p\le p_S(3+\mu)$, then
\begin{align*}
T(\varepsilon) \le 
\left\{
\begin{array}{ll}
\exp (C\varepsilon^{-p(p-1)}) &  \mbox{if} \hspace{3mm} p=p_S(3+\mu), \\
C\varepsilon^{-2p(p-1)/ \gamma_S(p,3+\mu)}
& \mbox{if} \hspace{3mm} 1<p <p_S(3+\mu).
\end{array}
\right.
\end{align*}
Here $T(\varepsilon)$ denotes the lifespan of the problem \eqref{eq.io0}.
\end{thm}

On the other hand, when $p>p_S(3+\mu)$, we prove that the solution exists globally,
analogously to \cite{GKW19}.
But the pointwise estimate for the solution is improved as in \eqref{pe1} below,
by refining the basic estimates.
In fact, the decaying order is stronger than the one in \cite{GKW19} 
in the region away from the light cone, due to the factor $\langle t+r \rangle^{-1}$.

To state our results, we introduce the following parameters: 
\begin{align}
\nu&=\kappa-\mu/2-1,\label{dai98}\\
\eta&=(\mu/2+1)p-(\mu/2+2). \label{dai89} 
\end{align}
Besides, $p_S(n)$ denotes the positive root of \eqref{gSdef}
and we put $p_F(\kappa):=1+2/\kappa$.

\begin{thm} \label{ge1}
Let $\kappa>\mu/2$.
Suppose that \eqref{1.3} holds and 
$f_0 \in C^1([0,\infty))$, $f_1 \in C^0([0,\infty))$
satisfy
\begin{align} \label{eq.zj1d}
|f_0(r)| \le \langle r \rangle^{-\kappa}, \quad
|f_0^\prime (r)| +|f_1(r)| \le \langle r \rangle^{-\kappa-1} \ \ \mbox{for} \ r \ge 0.
\end{align}
If $p > p_S(3+\mu)$ and $p \geq p_F(\kappa)$, then there exists $\varepsilon_0>0$ so that
for any $\varepsilon \in (0,\varepsilon_0]$ the corresponding integral equation
\eqref{eq.io0} to
the problem \eqref{1.6} has a unique global solution satisfying
\begin{align}\label{pe1}
|u(t,r)| \lesssim \! \varepsilon r \, \langle r \rangle^{-\mu/2}
\times
\left\{
\begin{array}{ll}
\langle t+r \rangle^{-1-\nu} & (-1<\nu<0),\\
\langle t+r \rangle^{-1}  \Psi(t-r,t+r)
& (\nu=0),\\
\langle t+r \rangle^{-1}  \langle t-r \rangle^{-\min \{ \nu,\eta \}} 
& (\nu>1)
\end{array}
\right.
\end{align}
for $t>0$, $r>0$. Here, for $|\beta|\leq \alpha$, we put
\begin{align}
 \Psi(\beta,\alpha):=1+\log \frac{1+\alpha}{1+|\beta|}.\label{dai96}
\end{align}
\end{thm}

When either $1<p\le p_S(3+\mu)$ or $1<p<p_F(\kappa)$,
we obtain the following lower bounds of the lifespan.

\begin{thm} \label{bw3}
Let $\kappa>\mu/2$. 
Suppose that \eqref{1.3} holds.
Let $f_0 \in C^1([0,\infty))$, $f_1 \in C^0([0,\infty))$
satisfy \eqref{eq.zj1d}.
If $1< p \leq p_S(3+\mu)$ or $1 < p < p_F(\kappa)$, then there exist $C>0$ and $\varepsilon_0>0$ such
that for any $\varepsilon\in (0,\varepsilon_0]$
\begin{align}
T(\varepsilon)\geq
\left\{
\begin{array}{ll}
\exp(C \varepsilon^{-p(p-1)}) & (p=p_S(3+\mu) \ \mbox{and} \ p \nu>1), \\ 
C\varepsilon^{-2p(p-1)/\gamma_S(p,3+\mu)} & (1<p<p_S(3+\mu) \ \mbox{and} \ p \nu>1), \\ 
\exp(C \varepsilon^{-(p-1)}) & (p=p_S(3+\mu)  \ \mbox{and} \ p \nu=1),\\
Cb(\varepsilon) & (1<p<p_S(3+\mu) \ \mbox{and} \ p \nu=1), \\ 
C\varepsilon^{-(p-1)/\gamma_F(p,\kappa)} & (1<p<p_F(\kappa) \ \mbox{and} \ p\nu<1). 
\end{array}
\right.\label{llifespan}
\end{align}
Here we put $\gamma_F(p,\kappa)=2-(p-1)\kappa$, and $b(\varepsilon)$ is defined by
\begin{align}
	 b(\varepsilon)
        \left\{\log(1+ b(\varepsilon) )\right\}^{2(p-1)/\gamma_S(p,3+\mu)}
      =\varepsilon^{-2p(p-1)/\gamma_S(p,3+\mu)}.
\label{bdef}
\end{align}
\end{thm}

We remark that $b(\varepsilon)$ is well-defined, because $\gamma_S(p,3+\mu)>0$
for $1<p<p_S(3+\mu)$.
Moreover, it is easy to see that $b(\varepsilon)$ is a decreasing function 
and tends to $\infty$ as $\varepsilon\to 0+0$. 
Also, we notice that the last case where $1<p<p_F(\kappa)$ and $p\nu<1$
includes the case where $1<p\le p_S(3+\mu)$ and $p\nu<1$.

To conclude the optimality of the lower bounds with respect to $\varepsilon$,
the upper bounds given in Theorem \ref{bw2} are not enough for the last three cases.
Indeed, 
both $b(\varepsilon)$ and $\varepsilon^{-(p-1)/\gamma_F(p,\kappa)}$
is smaller than $\varepsilon^{-2p(p-1)/\gamma_S(p,3+\mu)}$,
if we choose $\varepsilon$ is suitably small
and $p \nu \le 1$, because
$$
\frac{\gamma_S(p,3+\mu)}{2p(p-1)}=\frac{\gamma_F(p,\kappa)}{p-1}+\frac{p\nu -1}{p}.
$$
However, the following result tells us the optimality in those cases.

\begin{thm} \label{bw1}
Suppose that \eqref{1.3} holds.
Let $\varphi$, $\psi \in C([0,\infty))$ satisfy
\begin{align} \label{hyp1}
\varphi (r)\equiv 0, \quad \psi(r) \ge (1+r)^{-\kappa}
\quad \mbox{for}\ r\ge 0
\end{align}
for some $\kappa>0$. 
If either $1<p\le p_S(3+\mu)$ and $p \nu=1$, or $1<p<p_F(\kappa)$
and $p \nu<1$, 
then there exist 
$C>0$ and $\varepsilon_0>0$ such that for any $\varepsilon\in (0,\varepsilon_0]$ 
\begin{align*}
T(\varepsilon)\le
\left\{
\begin{array}{ll}
\exp(C \varepsilon^{-(p-1)}) & (p=p_S(3+\mu)  \ \mbox{and} \ p \nu=1), \\
Cb(\varepsilon) & (1<p<p_S(3+\mu) \ \mbox{and} \ p \nu=1), \\ 
C\varepsilon^{-(p-1)/\gamma_F(p,\kappa)} & (1<p<p_F(\kappa) \ \mbox{and} \ p \nu<1). 
\end{array}
\right.
\end{align*}
\end{thm}

This paper is organized as follows.
In the section 3, we give preliminary facts.
In the section 4, we derive a priori upper bounds and complete the proofs of Theorem \ref{ge1} 
and Theorem \ref{bw3}.
In particular, it is essential to linearize the problem around the free solution
from the nonlinearity as \eqref{dai92}.
The section 5 is devoted to the proof of a blow-up result given in Theorem \ref{bw1}
based on the argument due to John \cite{Joh79}.

\section{Preliminaries}
In this section we prepare lemmas 
which will be used in the proofs of the theorems.
For the proofs of Lemma \ref{lem1} and Lemma \ref{lem3.1}, 
see \cite{KK22}, Lemma 3 and Lemma 4, respectively.
\begin{lm}\label{lem1}
Let $0\leq a \leq b$ and $k \in \mathbb{R}$. Then we have
\begin{align*}
	\int_a^b \langle x \rangle^{-k}dx
	\lesssim (b-a) \times
	\left\{
	\begin{array}{ll}
	\langle b \rangle^{-k} & (k<1),\\
	\langle b \rangle ^{-1} \Psi(a,b) & (k=1),\\
	\langle b \rangle^{-1} \langle a \rangle^{-k+1} & (k>1).
	\end{array}
	\right.
\end{align*}
\end{lm}
\begin{lm}\label{lem3.1}
Let $k_1, k_2, k_3 \geq 0$ and $\alpha \geq0$. Then we have
\begin{align*}
\int_{-\alpha}^{\alpha} \langle \alpha + \beta \rangle^{-k_1-k_2}
\jb{\beta}^{-k_1-k_3} d\beta
\lesssim
\jb{\alpha}^{-k_1}  \times
\left\{
\begin{array}{ll}
\Phi_{1-(k_1+k_2+k_3)}(\alpha)
& (k_1+k_2+k_3 \neq 1),\\
\log(2+\alpha) & (k_1+k_2+k_3 = 1).
\end{array}
\right.
\end{align*}
Here, we put
\begin{align}\label{dai104}
	\Phi_{\rho}(s):=\max \{1,\langle s \rangle^{\rho}\}
	\quad (s,\rho \in \mathbb{R}).
\end{align}
\end{lm}
\begin{lm}\label{lem1.1}
Let $0\leq a \leq b$ and $k > 1$. Then we have
\begin{align}
	\int_a^b (1+x)^{-k} \log(1+x) dx
	\lesssim (b-a) 
	\langle b \rangle^{-1} \langle a \rangle^{-k+1}\log (2+a).
	\label{dai25}
\end{align}
\end{lm}
\begin{proof}
Integrating by parts, we get
\begin{align*}
\int_a^b (1+x)^{-k} \log(1+x) dx
=&\frac{1}{1-k}\left[ (1+x)^{1-k} \log(1+x) \right]_{a}^{b}\\
&+\frac{1}{k-1} \int_{a}^{b}(1+x)^{-k} dx\\
\lesssim & \frac{\log(1+a)}{(1+a)^{k-1}}
\left\{1- \left( \frac{1+a}{1+b} \right)^{k-1} \right\}\\
&+\int_{a}^{b}(1+x)^{-k} dx.
\end{align*}
Using the inequality:
\begin{align*}
1-s^l \leq \max \{ 1,l \} (1-s) \quad \mbox{for} \quad l\geq0,\ 0 \leq s \leq 1
\end{align*}
in the first term, and Lemma \ref{lem1} in the second term, we get \eqref{dai25}.
This completes the proof.
\end{proof}
\begin{lm}\label{lem2}
Let $0 \leq a \leq \alpha$ and $k \geq 0$. 
Then we have
\begin{align*}
	f_{k}(\alpha,a):=\int_{a}^{\alpha}
	\left\{\Psi(\beta,\alpha)
	\right\}^{k} d\beta \leq C_{k}(\alpha-a),
\end{align*}
where $\Psi(\beta, \alpha)$ is defined in \eqref{dai96}.
\end{lm}
\begin{proof}
It suffices to prove the estimate when $k$ is a non-negative integer, because $\Psi(\beta,\alpha) \geq 1$ for $a \leq \beta \leq \alpha$. It is clear that $f_0(\alpha,a)=\alpha-a$. Suppose that $f_{k-1}(\alpha,a) \leq C_{k-1}(\alpha-a)$ holds. Then, by the
integration by parts, we have
\begin{align*}
f_{k}(\alpha,a)
&=\left[\beta  \left\{1+\log \left( \frac{1+\alpha}{1+\beta} \right)
	\right\}^{k} \right]_{a}^{\alpha}\\
	&+k
	\int_{a}^{\alpha}\frac{\beta}{1+\beta}
	\left\{1+\log \left( \frac{1+\alpha}{1+\beta} \right)
	\right\}^{k-1} d\beta\\
	&\leq \alpha-a +k f_{k-1}(\alpha,a).
\end{align*}
Thus we get the desired estimate by the induction.
This completes the proof.
\end{proof}
\begin{lm}\label{lem3}
Let $r_1 \neq 1$, $r_2 \geq 0$ and $\alpha \geq 0$. Then we have
\begin{align*}
\int_{-\alpha}^{\alpha}
	\langle \alpha+ \beta \rangle^{-r_1} \left\{ \Psi(\beta,\alpha)
	\right\}^{r_2} d\beta 
	\lesssim  \Phi_{1-r_1}(\alpha).
\end{align*}
\end{lm}
\begin{proof}
We set
\begin{align*}
	I_1&:=\int_{-\alpha}^{-\alpha/2} \langle \alpha + \beta \rangle^{-r_1} \left\{\Psi(\beta, \alpha) \right\}^{r_2} d\beta,\\
	I_2&:=\int_{-\alpha/2}^{\alpha} \langle \alpha + \beta \rangle^{-r_1} \left\{ \Psi(\beta, \alpha) \right\}^{r_2} d\beta.
\end{align*}
We have from \eqref{dai96}
\begin{align*}
I_{1} \lesssim (1+ \log 2)^{r_2} \int_{-\alpha}^{\alpha} \langle \alpha + \beta \rangle^{-r_1} d\beta
	\lesssim  \Phi_{1-r_1}(\alpha).
\end{align*}
From Lemma \ref{lem2}, we obtain
\begin{align*}
I_2 \lesssim \langle \alpha \rangle^{-r_1} \int_{-\alpha}^{\alpha} \left\{ \Psi(\beta, \alpha)
	\right\}^{r_2} d\beta
	\lesssim  \langle \alpha \rangle^{1-r_1}.
\end{align*}
Combining these estimates, we finish the proof.
\end{proof}

For the proofs of Lemma \ref{KO} and Lemma \ref{lem:boi-1}, 
see \cite{KubOht06}, Lemma 2.2 and Lemma 2.3, respectively.

\begin{lm} \label{KO}
Let $\rho_1 \geq 0$ and  $\rho_2 \ge 0$. Then there exists a constant
$C=C(\rho_1, \rho_2)>0$ such that
\begin{align*}
\int_a^b\frac{(y-a)^{\rho_2}}{y^{\rho_1}}d y
\ge \frac{C}{a^{\rho_1-\rho_2-1} }\left(1-\frac{a}{b}\right)^
{\rho_2+1}
\end{align*}
for $0< a< b$.
\end{lm}

\begin{lm} \label{lem:boi-1}
Let $C_1$, $C_2>0$, $\alpha$, $\beta\ge 0$, $\theta \le 1$,
$\varepsilon\in (0,1]$, and $p>1$.
Suppose that $f(y)$ satisfies
$$
f(y)\ge C_1\varepsilon^{\alpha},\quad
f(y)\ge C_2\varepsilon^{\beta}\int_{1}^{y}\left(1-\frac{\xi}{y}\right)
\frac{f(\xi)^{p}}{\xi^{\theta}}\, d\xi,\quad y\ge 1.
$$
Then, $f(y)$ blows up in a finite time $T_*(\varepsilon)$. Moreover,
there exists a constant $C^*=C^*(C_1,C_2,p,\theta)>0$ such that
$$T_*(\varepsilon)\le \left\{ \begin{array}{ll}
\exp(C^*\varepsilon^{-\{(p-1)\alpha+\beta\}})
& \mbox{if} \hspace{2mm} \theta=1, \\
C^*\varepsilon^{-\{(p-1)\alpha+\beta\}/(1-\theta)}
& \mbox{if} \hspace{2mm} \theta <1.
\end{array} \right.
$$
\end{lm}

\section{Global existence and lower bounds of the lifespan}

\subsection{Proof of Theorem \ref{ge1} and  Theorem \ref{bw3} for $1<p<p_F(\kappa)$
and $p \nu<1$}

Our first step is to obtain the estimates for the homogeneous part of the solution
to the problem \eqref{eq.io0}.\\
We put 
\begin{align}
w_1(t,r)= r\, \langle r \rangle^{-\mu/2}
\times \left\{
\begin{array}{ll}
\langle t+r \rangle^{-1-\nu} & (\nu<0),\\
\langle t+r \rangle^{-1} \Psi(t-r,t+r)
& (\nu=0),\\
\langle t+r \rangle^{-1}  \langle t-r \rangle^{-\nu} 
& (\nu>0),
\end{array}
\right.\label{weight1}
\end{align}
where $\nu$ and $\Psi(\beta,\alpha)$ are defined in \eqref{dai98} and \eqref{dai96}, respectively.

We define the following weighted $L^\infty$-norm:
\begin{align}
\| u \|_1 =\sup_{ (t,r) \in I \times [0,\infty)}
\{ w_{1}(t,r)^{-1}
|u(t,r)| \},\label{norm1}
\end{align}
where we put
\begin{align*}
	I:=\left\{
	\begin{array}{ll}
	[0,\infty) & (p>p_{S}(3+\mu) \ \mbox{and} \ p \geq p_{F}(\kappa)),\\ \relax
	[0,T) & (1<p \leq p_{S}(3+\mu) \ \mbox{or} \ 1<p < p_{F}(\kappa)).
	\end{array}
	\right.
\end{align*}

For the proof of Lemma \ref{fil1}, see \cite{KK22}, Lemma 6.

\begin{lm} \label{fil1}
Assume that \eqref{1.3} holds and $ \varphi \in C^1([0,\infty))$, $ \psi \in C^0([0,\infty))$ satisfy
\begin{align*} 
|\varphi(r)| \lesssim \, r \, \langle r \rangle^{-\kappa}, \quad
|\psi(r)+\varphi^\prime(r)+w(r)\varphi(r) | \lesssim  \langle r \rangle^{-\kappa} \ \mbox{for} \ r \ge 0
\end{align*}
with some positive constant $\kappa$.
Then we have
\begin{align} \label{eq.io2}
| u_L (t,r)| \leq  \widetilde{C_0} w_1(t,r), \quad (t,r) \in [0,\infty) \times [0,\infty)
\end{align}
with some positive constant $\widetilde{C_0}$. Here $u_L$ is defined in \eqref{eq.io00}
\end{lm}
Our next step is to consider the integral operator appeared in \eqref{eq.io0}:
\begin{align*} 
I_-(F)(t,r)=
\frac12 \iint_{\Delta_-(t,r)} E_-(t-\sigma,r,y) F(\sigma, y) dy d\sigma.
\end{align*}
For $(\sigma,y) \in \Delta_{-}(t,r)$ we have $y \geq |t-r-\sigma|$, so that \eqref{eq.le2} yields
\begin{align*}
	|E_{-}(t-\sigma,r,y)| \lesssim \jb{r}^{-\mu/2} \jb{y}^{\mu/2}
	\quad \mbox{for} \quad (\sigma,y)\in \Delta_{-}(t,r).
\end{align*}
Hence we get
\begin{align} \label{eq.io1.1}
|I_-(F)(t,r)| \lesssim \jb{r}^{-\mu/2}
\iint_{\Delta_-(t,r)}  \jb{y}^{\mu/2} |F(\sigma, y)| dy d\sigma.
\end{align}
We set
\begin{align}\label{dai99}
D_1(T)=\left\{
\begin{array}{ll}
1 & (p \geq p_{F}(\kappa)),\\
	(1+T)^{2-(p-1)\kappa} & (1<p<p_{F}(\kappa)).
\end{array}
\right.
\end{align}
In order to prove Theorem \ref{ge1} and Theorem \ref{bw3} with $p \nu <1$, we prepare the following Lemma \ref{lem4}
and Lemma \ref{lem6}.
\begin{lm}\label{lem4}
Let $-1< \nu <0$ and $p>1$. Then, there exists a positive constant $\widetilde{C_1}$ such that
\begin{align}
	\|I_{-}(F)\|_{1} \leq \widetilde{C_1} \| u \|_1^p D_1(T) \label{dai1}
\end{align}
with $F(t,r)= |u(t,r)|^{p}/r^{p-1}$.
\end{lm}
\begin{proof}
From \eqref{weight1} and \eqref{norm1}, we obtain
\begin{align*}
	\langle r \rangle^{p \mu/2 } 
	\langle t+r \rangle^{p(1+\nu)}|F(t,r)|
	\leq r \| u \|_{1}^{p}.
\end{align*}
It follows from (\ref{eq.io1.1}) that
\begin{align*}
	|I_{-}(F)(t,r)| \lesssim \langle r \rangle^{-\mu/2} \| u \|_{1}^p I(t,r),
\end{align*}
where we put
\begin{align}
I(t,r)= \iint_{\Delta_{-}(t,r)}
\frac{y}{\langle y \rangle^{(p-1)\mu/2}
\langle \sigma+y \rangle^{p(1+\nu)}} dy d\sigma.\label{dai3}
\end{align}
From \eqref{norm1}, \eqref{weight1} and \eqref{dai99}, in order to \eqref{dai1}, it is enough to prove 
\begin{align}
	I(t,r) \lesssim \frac{r}{\langle t+r \rangle^{1+\nu}}
	\times \left\{
	\begin{array}{ll}
	1 & (p\geq p_{F}(\kappa)),\\
	(1+t)^{2-(p-1)\kappa} & (1<p<p_{F}(\kappa)).
	\end{array}
	\right.\label{dai4}
\end{align}
To evaluate the integral (\ref{dai3}), we pass to the coordinates
\begin{align*}
	\alpha=\sigma+y,\quad z=y
\end{align*}
and deduce
\begin{align}
	I(t,r) \lesssim \int_{|r-t|}^{t+r} 
	\frac{1}{\langle \alpha \rangle^{p(1+\nu)}}
	\int_{(\alpha-t+r)/2}^{\alpha}
	\frac{z}{\langle z \rangle^{(p-1)\mu/2}} dz d\alpha.
	\label{dai5}
\end{align}

First, we suppose $t \geq r/2$ and divide the proof into two cases.\\
(i) \ $p\neq p_{F}(\kappa)$.\\
First of all, we note that \eqref{dai98} implies
\begin{align}\label{dai105}
	-(p-1)\mu/2=(p-1)(1+\nu)-(p-1)\kappa.
\end{align}
Since $(p-1)(1+\nu)>0$ and $(p-1)\kappa \neq 2$, we get from (\ref{dai5}), \eqref{dai105}
and \eqref{dai104}
\begin{align*}
	I(t,r) &\lesssim \int_{|r-t|}^{t+r} 
	\frac{1}{\langle \alpha \rangle^{1+\nu}}
	\int_{0}^{\alpha} \frac{1}{\langle z \rangle^{(p-1)\kappa-1}}
	dz d\alpha\\
	&\lesssim \int_{|r-t|}^{t+r} 
	\frac{1}{\langle \alpha \rangle^{1+\nu}}
	\Phi_{2-(p-1)\kappa}(\alpha) d\alpha\\
	&\lesssim \Phi_{2-(p-1)\kappa}(t)
	\int_{|r-t|}^{t+r} 
	\frac{1}{\langle \alpha \rangle^{1+\nu}}
	d\alpha.
\end{align*}
Here, $\Phi_{\rho}(s)$is defined in \eqref{dai104}.	
Since $\nu<0$, it follows from Lemma \ref{lem1} that
\begin{align}
I(t,r)
&\lesssim \frac{r}{\langle t+r \rangle^{1+\nu}} \Phi_{2-(p-1)\kappa}(t).
\label{dai6}
\end{align}
(ii) \ $p = p_{F}(\kappa)$.\\
Taking $\delta>0$ so small that
$(p-1)(1+\nu)-\delta>0$.
Since $(p-1)\kappa=2$, from (\ref{dai5}), \eqref{dai105} and Lemma \ref{lem1}, we have
\begin{align}
	I(t,r) &\lesssim \int_{|r-t|}^{t+r} 
	\frac{1}{\langle \alpha \rangle^{1+\nu+\delta}}
	\int_{0}^{\alpha} \frac{1}{\langle z \rangle^{1-\delta}} dz
	d \alpha\nonumber\\
	&\lesssim \int_{|r-t|}^{t+r} 
	\frac{1}{\langle \alpha \rangle^{1+\nu}}
	d \alpha\nonumber\\
	&\lesssim \frac{r}{\langle t+r \rangle^{1+\nu}}.\label{dai7}
\end{align}
Therefore, from (\ref{dai6}) and (\ref{dai7}), we obtain (\ref{dai4}) for $t \geq r/2$.

Next, suppose $r/2 \geq t$. Since $(r+t)/3\leq r-t$, we obtain
from (\ref{dai5}) and \eqref{dai105}
\begin{align*}
I(t,r) 
&\lesssim \int_{r-t}^{t+r} \frac{\alpha (\alpha+t-r)}{\langle \alpha \rangle^{p(1+\nu)} 
\langle \alpha-t+r \rangle^{(p-1)\mu/2}} d\alpha\\
&\lesssim
\frac{t+r}{\langle t+r \rangle^{1+\nu+(p-1)\kappa}} 
\int_{r-t}^{t+r} (\alpha+t-r) d\alpha\\
&\lesssim \frac{r}{\langle t+r \rangle^{1+\nu}}
(1+t)^{2-(p-1)\kappa}.
\end{align*}
Hence we get (\ref{dai4}) for $r/2 \geq t$.
This completes the proof.
\end{proof}
\begin{lm}\label{lem6}
	Let $0 \leq \nu < 1/p$ if $1<p\leq p_S(3+\mu)$ and $0 \leq \nu \leq \eta$ if $p>p_S(3+\mu)$. Then, there exists a positive constant $\widetilde{C_1}$ such that
	\begin{align}
	\|I_{-}(F)\|_{1} \leq \widetilde{C_1} \| u \|_1^p \times
	D_1(T) \label{dai8}
\end{align}
with $F(t,r)= |u(t,r)|^{p}/r^{p-1}$.
Here $\eta=(\mu/2+1)p-(\mu/2+2)$.
\end{lm}
\noindent
{\bf Remark}. \
When $p>p_S(3+\mu)$ (resp. $1<p<p_S(3+\mu)$), we have
$\eta>1/p$ (resp. $\eta<1/p$).
\begin{proof}
It follows from \eqref{weight1} and  \eqref{norm1} that
\begin{align*}
	\langle r \rangle^{p \mu/2 } 
	\langle t+r \rangle^{p}\langle t-r \rangle^{p \nu}|F(t,r)|
	\leq r 
	\left\{ \Psi(t-r,t+r) \right\}^{p q_1}
	\| u \|_{1}^{p},
\end{align*}
where the number $q_1$ is defined as follows:
$$q_1=1 \ \mbox{for }\ \nu=0, \ \mbox{and} \ q_1=0 \ \mbox{for} \ \nu \neq 0.
$$ 
We get from (\ref{eq.io1.1})
\begin{align*}
	|I_{-}(F)(t,r)|\lesssim \langle r \rangle^{-\mu/2} \| u \|_{1}^p  I(t,r),
\end{align*}
where we put
\begin{align}
I(t,r)= \iint_{\Delta_{-}(t,r)}
\frac{y \left\{ \Psi(\sigma-y,\sigma+y) \right\}^{p q_1}}{\langle y \rangle^{(p-1)\mu/2}
\langle \sigma+y \rangle^{p}
\langle \sigma-y \rangle^{p \nu}} dy d\sigma.
\label{dai10}
\end{align}
In order to show \eqref{dai8}, from \eqref{weight1} and  \eqref{norm1}, it is enough to show
\begin{align}
	I(t,r)\lesssim \frac{r  \left\{ \Psi(t-r,t+r) \right\}^{q_1}}{
	\langle t+r \rangle
	\langle t-r \rangle^{\nu}}
	\times \left\{
	\begin{array}{ll}
	1 & (p\geq p_{F}(\kappa)),\\
	(1+t)^{2-(p-1)\kappa} & (1<p<p_{F}(\kappa)).
	\end{array}
	\right.\label{dai11}
\end{align}

To evaluate the integral (\ref{dai10}), we pass to the coordinates
\begin{align*}
\alpha = \sigma+y,\quad \beta=y-\sigma,
\end{align*}
and deduce
\begin{align}
	I(t,r) &\lesssim \int_{|r-t|}^{t+r} 
	\int_{r-t}^{\alpha}
	\frac{\left\{ \Psi(\beta,\alpha) \right\}^{p q_1}}
	{
	\langle \alpha+\beta \rangle^{(p-1)\mu/2-1}
	\langle \alpha \rangle^{p}
	\langle \beta \rangle^{p \nu}} d\beta d\alpha.
	\label{dai20}
\end{align}

First, suppose $r \geq t$. 
Then, we get from \eqref{dai20} and \eqref{dai89}
\begin{align}
I(t,r)
&\lesssim \int_{r-t}^{t+r} 
\frac{1}{
\langle \alpha \rangle^{1+\eta}}
\int_{r-t}^{\alpha} 
\frac{\left\{ \Psi(\beta,\alpha) \right\}^{p q_1}}{\langle \beta \rangle^{p\nu}} d\beta d\alpha.\label{dai30}
\end{align}
For $\alpha \geq r-t \geq 0$, we have from Lemma \ref{lem1} and Lemma \ref{lem2}
\begin{align}\label{dai12}
\int_{r-t}^{\alpha} \frac{\left\{ \Psi(\beta,\alpha) \right\}^{p q_1}}{\langle \beta \rangle^{p\nu}} d\beta
\lesssim \left\{
\begin{array}{ll}
\dfrac{\alpha-r+t}{\langle \alpha \rangle^{p\nu}} & (p \nu<1),\\
1+\log(1+\alpha) & (p \nu=1),\\
1 & (p \nu >1).
\end{array}
\right.
\end{align}
We divide the proof into two cases.\\
(i) \ $1/p \leq \nu \leq \eta$ and $p>p_S(3+\mu)$.\\
Since $\nu<\eta$ if $p \nu=1$,
we obtain from \eqref{dai30}, \eqref{dai12} and Lemma \ref{lem1}
\begin{align*}
	I(t,r) 
	\lesssim \int_{r-t}^{t+r} \frac{1}
	{\jb{\alpha}^{1+\nu}}d\alpha
	\lesssim \frac{r}{\langle t+r \rangle 
\langle t-r \rangle^{\nu}}.
\end{align*}
(ii) \ $0 \leq \nu <1/p$.\\
Since $\eta+p\nu=\nu+(p-1)\kappa-1$, 
we have from \eqref{dai30} and \eqref{dai12}
\begin{align}\label{dai18}
I(t,r) 
	&\lesssim  \int_{r-t}^{t+r} 
	\frac{\alpha-r+t}{\langle \alpha \rangle^{\nu+(p-1)\kappa}}d\alpha.
\end{align}
If $(p-1)\kappa\geq 2$, it follows from \eqref{dai18} and  Lemma \ref{lem1} that
\begin{align}
	I(t,r) 
	&\lesssim  \int_{r-t}^{t+r} 
	\frac{1}{\langle \alpha \rangle^{\nu+(p-1)\kappa-1}}d\alpha\nonumber\\
	&\lesssim \frac{r}{\langle t+r \rangle
	\langle r-t \rangle^{\nu}}
	\times
	\left\{
	\begin{array}{ll}
	1 & ((p-1)\kappa>2 \ \mbox{or} \ \nu>0),\\
	\Psi(t-r,t+r) & ((p-1)\kappa=2 \ \mbox{and} \ \nu=0).
	\end{array}
	\right.\label{dai13}
\end{align}
On the other hand, if $0 < (p-1)\kappa <2$, we obtain from \eqref{dai18} and Lemma \ref{lem1}
\begin{align}
	I(t,r) 
	&\lesssim \frac{1}{\langle r-t \rangle^{\nu}}  \int_{r-t}^{t+r} 
	\frac{\alpha-r+t}{\langle \alpha \rangle^{\kappa(p-1)}}d\alpha\nonumber\\
	&\lesssim \langle r-t \rangle^{-\nu} \times
	\left\{
	\begin{array}{ll}
	t \langle t+r \rangle^{1-(p-1)\kappa} & (1 \leq  (p-1)\kappa<2),\\
	t^2 \langle t+r \rangle^{-(p-1)\kappa} & (0< (p-1)\kappa<1).
	\end{array}
	\right.
	\label{dai14}
\end{align}
Let $r \geq 1$.
Since $r \sim \langle t+r \rangle$ for $r \geq t$, we get from (\ref{dai14})
\begin{align}
	I(t,r) 
	&\lesssim \frac{r}{\langle t+r \rangle \langle t-r \rangle^{\nu}}(1+t)^{2-(p-1)\kappa}.
	\label{dai15}
\end{align}
Let $0\leq r \leq 1$. Since $\langle t+r \rangle \sim 1$ for $r \geq t$, we have from (\ref{dai14})
\begin{align}
	I(t,r) 
	& \lesssim \frac{r}{\langle t+r \rangle
	\langle t-r \rangle^{\nu}}.\label{dai16}
\end{align}
Hence, from (\ref{dai13}), (\ref{dai15}) and (\ref{dai16}), we obtain (\ref{dai11}) for $r \geq t$.

Next, we suppose $t \geq r$ and divide the proof into three cases.
We remark that those cases assure the assumptions on $\nu$, because
$1<p\le p_S(3+\mu)$ and $p \nu<1$ implies $p<p_F(\kappa)$.\\
(i) \  $0 \leq \nu \leq \eta$ and $p \neq p_F(\kappa)$.\\
Since $\langle\alpha\rangle \ge \langle (\alpha+\beta)/2\rangle$ for
$-\alpha\le \beta \le \alpha$, we have from \eqref{dai20} and \eqref{dai89}
\begin{align}
I(t,r) \lesssim \int_{t-r}^{t+r} 
	\frac{1}{\langle \alpha \rangle}
	\int_{-\alpha}^{\alpha}
	\frac{\left\{ \Psi(\beta,\alpha) \right\}^{p q_1}}
	{ \langle \alpha+\beta \rangle^{\eta}
	\langle \beta \rangle^{p \nu}} d\beta d\alpha.
	\nonumber
\end{align}
We shall use Lemma \ref{lem3} when $\nu=0$, and Lemma \ref{lem3.1}
with $k_1=\nu$, $k_2=\eta-\nu$, $k_3=(p-1)\nu$ when $\nu > 0$.
Then we get
\begin{align}
I(t,r) \lesssim \int_{t-r}^{t+r} 
	\frac{1}{\langle \alpha \rangle^{1+\nu}} \Phi_{2-(p-1)\kappa}(\alpha)
	 d\alpha,\nonumber
\end{align}
because $(p-1)\kappa \neq 2$ when $p\ne p_F(\kappa)$, and 
\begin{align}
 \eta+(p-1)\nu=-1+(p-1)\kappa.
 \label{dai21}
\end{align}
Now, Lemma \ref{lem1} yields
\begin{align}
I(t,r) \lesssim \frac{r \left\{ \Psi(t-r,t+r) \right\}^{q_1}}{\langle t+r \rangle \langle t-r \rangle^{\nu}} \Phi_{2-(p-1)\kappa}(t),\label{dai22}
\end{align}
(ii) \ $0 \le \nu \le \eta$, $p=p_F(\kappa)$ and $p>p_S(3+\mu)$.\\
Note that $p=p_F(\kappa)$ and $p>p_S(3+\mu)$ leads to $\nu<1/p<\eta$.
Therefore, we can take $\delta>0$ so small that $p-1-\delta>0$ and 
$\eta-\delta>\nu$.
Then, similarly to the previous case, 
we get from (\ref{dai20}), \eqref{dai21}, Lemma \ref{lem3} 
and Lemma \ref{lem3.1}
\begin{align}
I(t,r) &\lesssim  \int_{t-r}^{t+r} 
	\frac{1}
	{\langle \alpha \rangle^{1+\delta}}
	\int_{-\alpha}^{\alpha}
	\frac{ \left\{ \Psi(\beta,\alpha) \right\}^{p q_1}}
	{ \langle \alpha+\beta \rangle^{\eta-\delta}
	\langle \beta \rangle^{p \nu}} d\beta d\alpha
	\nonumber\\
	& \lesssim \int_{t-r}^{t+r} 
	\frac{1}{\langle \alpha \rangle^{\nu+1+\delta}} \Phi_{\delta}(\alpha)
	 d\alpha\nonumber\\
	 & \lesssim \frac{r \left\{ \Psi(t-r,t+r) \right\}^{ q_1}}{\langle t+r \rangle \langle t-r \rangle^{\nu}}.\label{dai23}
\end{align}
(iii) \ $\max\{\eta, 0\} \leq \nu <1/p$  and $1<p\le p_S(3+\mu)$.\\
Since $p-1-\eta+\nu>0$ and $(p-1)(\mu/2+1)=1+\eta$, we get from (\ref{dai20})
\begin{align}
I(t,r) \lesssim \int_{t-r}^{t+r} \frac{1}{\langle \alpha \rangle^{1+\eta-\nu}}
\int_{-\alpha}^{\alpha}
\frac{\left\{ \Psi(\beta,\alpha) \right\}^{p q_1}}
{\langle \alpha+\beta \rangle^{\nu} \langle \beta \rangle^{p \nu}}d\beta d\alpha.
\nonumber
\end{align}
Using Lemma \ref{lem2} when $\nu=0$, and Lemma \ref{lem3.1} with
$k_1=\nu$, $k_2=0$, $k_3=(p-1) \nu$ when $\nu > 0$, 
we obtain
\begin{align}
I(t,r) &\lesssim \int_{t-r}^{t+r} \frac{1}{\langle \alpha \rangle^{\eta+p \nu}} d\alpha
=
 \int_{t-r}^{t+r} \frac{1}{\langle \alpha \rangle^{\nu-1+(p-1)\kappa}} d\alpha,
\nonumber
\end{align}
by $p\nu<1$ and \eqref{dai21}.
Recalling the fact that $(p-1)\kappa<2$ in this case, we find 
\begin{align}
I(t,r) &\lesssim \langle t+r \rangle^{2-\kappa(p-1)} \int_{t-r}^{t+r} \frac{1}{\langle \alpha \rangle^{\nu +1}} d\alpha\nonumber\\
&\lesssim \frac{r \left\{ \Psi(t-r,t+r) \right\}^{q_1}}{\langle t+r \rangle \langle t-r \rangle^{\nu}}(1+t)^{2-(p-1)\kappa}.
\label{dai24}
\end{align}
Hence, from (\ref{dai22}), (\ref{dai23}) and (\ref{dai24}), we obtain (\ref{dai11}) for $t \geq r$.
This completes the proof.
\end{proof}

\noindent
{\bf End of the proof of Theorem \ref{ge1} and 
Theorem \ref{bw3} for $1<p<p_F(\kappa)$ and $p \nu<1$}.
Let $X$ be the linear space defined by
\begin{align*}
X=\left\{  u(t,x) \in C(I \times (0,\infty)) \ ; \ \| u \|_1 < \infty \right\}.
\end{align*}
We can verify easily that $X$ is complete with respect  the norm $\| \cdot \|_1$. We define the sequence of functions $\{ u_n \}$ by
\begin{align*}
u_0=\varepsilon u_{L},\quad u_{n+1}= \varepsilon u_{L} +I_{-}(|u_{n}|^p/r^{p-1}) \quad (n=0,1,2,\cdots).
\end{align*}
It follows from Lemma \ref{fil1} that $\| u_0 \|_{1} \leq \varepsilon \widetilde{C}_0$. Hence $u_{0} \in X$. 

In the following, we assume that
\begin{align}
2^p p \widetilde{C}_{1} D_1(T) (\varepsilon \widetilde{C}_{0})^{p-1} \leq 1 \quad 
\mbox{and} \quad \varepsilon \widetilde{C}_0 \leq \frac{1}{2},\label{dai90}
\end{align}
where $\widetilde{C}_1$  is the constant in Lemma \ref{lem4} and Lemma \ref{lem6}.
Then, we have
\begin{align}
2^p p \widetilde{C}_{1} D_1(T) \| u_{0} \|_1^{p-1} \leq 1 \quad 
\mbox{and} \quad \| u_{0} \|_1 \leq \frac{1}{2}.\label{dai91}
\end{align}

Now, we conclude the proof of the theorems.
First, we prove Theorem \ref{ge1}.
Suppose that $\kappa >\mu/2$, $p>p_S(3+\mu)$ and $p \geq p_{F}(\kappa)$ hold.
Then $\nu>-1$ and $D_1(T)=1$.
In addition, we replace $\kappa$ by $\tilde{\kappa}=\min\{\kappa, (\mu/2+1)p-1\}$ 
in \eqref{eq.zj1d}, so that $\tilde{\nu}:=\tilde{\kappa}-\mu/2-1=\min\{\nu,\eta\}$.
As in \cite{Joh79}, we see that if $u_{0}$ satisfies (\ref{dai91}), then 
it follows from Lemma \ref{lem4} and Lemma \ref{lem6} that 
$\{ u_n\}$ is a Cauchy sequence in $X$. 
Since $X$ is complete, there exists $u \in X$ such that $u_{n}$ converges to $u$ uniformly.
Clearly $u$ satisfies \eqref{eq.io0}.
Thus we get Theorem \ref{ge1},

Next,  we move on to the proof of Theorem \ref{bw3} for the case
where $\kappa >\mu/2$,  $1<p<p_F(\kappa)$ and $p \nu<1$ hold.
In this case, we have $\nu>-1$ and $D_1(T)=(1+T)^{2-(p-1)\kappa}$.
Similarly to the above, we obtain a unique local solution of (\ref{eq.io0}),
provided (\ref{dai90}) holds.
This means (\ref{llifespan}) with $1<p<p_{F}(\kappa)$ and $p \nu <1$ is valid. 
This competes the proof.\\

\subsection{Proof of Theorem \ref{bw3} with $p \nu \geq 1$}

It is convenient to transform the equation \eqref{eq.io0} to
the following integral equation:
\begin{align}
v=I_{-}(|\varepsilon u_{L} +v|/r^{p-1}),\label{dai92}
\end{align}
as in the proof of Theorem 2.3 in \cite{KubOht06}.
We observe that the maximal existence time of 
the solution of \eqref{eq.io0} is the same as that of the solution $v=v(t,r)$ of \eqref{dai92}.
Indeed, if we would obtain a solution $v$ of \eqref{dai92}, then $u:=\varepsilon u_{L} +v$
is the solution of \eqref{eq.io0}, so that the desired conclusion follows.

In order to solve \eqref{dai92}, we introduce a weight function:
\begin{align}
w_{2}(t,r):=\frac{r}{\langle r \rangle^{\mu/2}} \times \left\{
\begin{array}{lr}
\dfrac{ \left\{ \Psi(t-r,t+r) \right\}^{q_2} \left\{\log(2+t+r) \right\}^{q_3} }{\langle t+r \rangle^{1+\eta}}
& (\eta \leq 0),\\
\dfrac{ \left\{\log(2+|t-r|) \right\}^{q_3} }{\jb{t+r}\langle t-r \rangle^{\eta}}
& (\eta >0),
\end{array}
\right.\nonumber\\
\label{weight2}
\end{align}
where we put
\begin{align*}
&q_2=\left\{
\begin{array}{ll}
1 & (\eta=0),\\
0 &  (\eta \neq 0),
\end{array}
\right. \quad
q_3=\left\{
\begin{array}{ll}
1 & (p \nu=1),\\
0 &  (p \nu >1).
\end{array}
\right.
\end{align*}
By using the following weighted $L^{\infty}$ norm:
\begin{align}\label{norm2}
	\| v \|_{2}:=\sup_{(t,r) \in [0,T) \times [0,\infty)} w_{2}(t,r)^{-1}|v(t,r)|,
\end{align}
we shall prove Lemma \ref{lem8} and Lemma \ref{lem9} below.
\begin{lm}\label{lem8}
If $p\nu \geq 1$ and $1<p \leq p_S(3+\mu)$, then there exists a positive constant $\widetilde{C}_1$ such that
\begin{align}
	\| I_{-}(|u_L|^p/r^{p-1}) \|_{2} \leq  \widetilde{C}_1 \| u_L\|_{1}^{p}.
	\label{dai45}
\end{align}
\end{lm}
\begin{proof}
Since $\nu>0$, similarly to the proof of Lemma \ref{lem6}, we get
$$
| I_{-}(|u_L|^p/r^{p-1})(t,r)| \lesssim 
	\langle r \rangle^{\mu/2} \| u_L\|_{1}^{p} I(t,r),
$$
where $I(t,r)$ is the one defined in \eqref{dai10}.
In view of \eqref{weight2} and \eqref{norm2}, 
for getting \eqref{dai45}, it is enough to prove
\begin{align}
	I(t,r)
	\lesssim 
	\langle r \rangle^{\mu/2} w_2(t,r).\label{dai48}
\end{align}
It follows from \eqref{dai20} and \eqref{dai89} that
\begin{align*}
	I(t,r)	&\lesssim \int_{|r-t|}^{t+r} 
	\frac{1}{\langle \alpha \rangle}
	\int_{-\alpha}^{\alpha}
	\frac{1}
	{\langle \alpha+\beta \rangle^{\eta}
	\langle \beta \rangle^{p \nu}} d\beta d\alpha.
\end{align*}
From 
$1<p \leq p_S(3+\mu)$, we have
$\eta \le 1/p$, so that $p\nu> \nu \ge 1/p \ge \eta$.
Therefore, using
Lemma \ref{lem3.1} with $k_1=\eta$, $k_2=0$, $k_3=p\nu-\eta$, we get
\begin{align*}
	\int_{-\alpha}^{\alpha}
	\frac{1}
	{\langle \alpha+\beta \rangle^{\eta}
	\langle \beta \rangle^{p \nu}} d\beta
	\lesssim \langle \alpha \rangle^{-\eta} 
	\left\{ \log(2+\alpha) \right\}^{q_3}.
\end{align*}
From Lemma \ref{lem1} and Lemma \ref{lem1.1}, we obtain
\begin{align*}
I(t,r) &\lesssim \int_{|t-r|}^{t+r}\frac{ \left\{ \log(2+\alpha) \right\}^{q_3}}{\jb{\alpha}^{1+\eta}}
d \alpha\\
&\lesssim r \times
\left\{
\begin{array}{ll}
\langle t+r \rangle^{-(1+\eta)}  \{ \Psi(t-r,t+r) \}^{q_2}  \left\{ \log(2+t+r) \right\}^{q_3}  & (\eta \leq 0),\\
\langle t+r \rangle^{-1} \langle t-r \rangle^{-\eta} \left\{ \log(2+|t-r|) \right\}^{q_3} & (\eta >0).
\end{array}
\right.
\end{align*}
Hence \eqref{dai48} is deduced. This completes the proof.
\end{proof}

We set
\begin{align}\label{defD2}
D_2(T)=\left\{
\begin{array}{ll}
	(1+T)^{\gamma_S(p,3+\mu)/2} 
	\left\{ \log(2+T) \right\}^{(p-1)q_3}
	& (1<p<p_S(3+\mu)),\\
	\left\{ \log(2+T) \right\}^{1+(p-1)q_3} & (p=p_S(3+\mu)).
\end{array}
\right.
\end{align}

\begin{lm}\label{lem9}
Let $p_1=1$ or $p$.
If $p \nu \geq 1$ and $1<p \leq p_S(3+\mu)$, then there exists a positive constant $\widetilde{C}_2$ such that
\begin{align}\label{dai103}
	\| I_{-}(G) \|_{2} \leq  \widetilde{C}_2 \| u_{L} \|_{1}^{p-p_1}
 \| v \|_2^{p_1}
	D_2(T)^{{p_1}/p}.
\end{align}
with $ G(t,r)=|u_L(t,r)|^{p-p_1}|v(t,r)|^{p_1}/r^{p-1}$.
\end{lm}
\begin{proof}
At first, we put
\begin{align*}
\eta_1=p_1\eta+(p-p_1) \nu, \quad 
\eta_2=p_1 \eta +p,\quad
\eta_3=(p-p_1)/p.
\end{align*}
Noting that $1-p\eta=\gamma_S(p,3+\mu)/2$, we get
\begin{align}
&1-\eta_1=\gamma_S(p,3+\mu) p_1/(2p)-(p-p_1)(p\nu-1)/p, \label{dai106}\\
&1-p_1 \eta -\eta_3=\gamma_S(p,3+\mu) p_1/(2p).  \label{dai106+}
\end{align}
Besides, \eqref{dai89} implies
\begin{align}
(p-1)\mu/2+\eta_2=2+(1+p_1)\eta.\label{dai107}
\end{align}
In the following, let $(t,r) \in [0,T)\times (0,\infty)$,
$p\nu \ge 1$, and either $p_1=1$ or $p_1=p$.
We divide the proof into two cases.
\\
(i) \ $\eta >0$.\\
Since $\nu >0$, we have from  \eqref{weight1}, \eqref{norm1} and \eqref{eq.io2}
\begin{align}
	 \langle r \rangle^{\mu/2} \langle t+r \rangle
	\langle t-r \rangle^{\nu}|u_L(t,r)| \leq r \| u_L \|_1.\label{dai109}
\end{align}
Therefore, we get from \eqref{weight2} and \eqref{norm2} 
\begin{align*}
	\langle r \rangle^{p \mu/2 } 
	\langle t+r \rangle^{p}
	\langle t-r \rangle^{\eta_1}|G(t,r)|
	\leq r  \left\{ \log(2+|t-r|) \right\}^{p_1 q_3} 
	\| u_L \|_1^{p-p_1} \| v \|_2^{p_1}.
\end{align*}
Then, it follows from (\ref{eq.io1.1}) that
\begin{align*}
	|I_{-}(G)(t,r)|\lesssim \langle r \rangle^{-\mu/2} 
	\| u_L \|_1^{p-p_1} \| v \|_2^{p_1}  I(t,r),
\end{align*}
where we put
\begin{align*}
I(t,r):= \iint_{\Delta_{-}(t,r)}
\frac{y \left\{ \log(2+|\sigma-y|)\right\}^{p_1 q_3}}{\langle y \rangle^{(p-1)\mu/2}
\langle  \sigma +y \rangle^{p}
\langle \sigma-y \rangle^{\eta_1}} dy d\sigma.
\end{align*}
In order to get \eqref{dai103}, by \eqref{weight2} and \eqref{norm2}, it is enough to prove
\begin{align}\label{dai51}
	I(t,r)\lesssim & \frac{r  \left\{ \log(2+|t-r|)\right\}^{q_3}}{
	\langle t+r \rangle
	\langle t-r \rangle^{\eta}}
	D_2(T)^{p_1/p}.
\end{align}
To evaluate the integral, we pass to the coordinates
\begin{align*}
\alpha = \sigma+y,\quad \beta=y-\sigma,
\end{align*}
and deduce
\begin{align}\label{dai52}
	I(t,r) &\lesssim \int_{|r-t|}^{t+r} 
	\int_{r-t}^{\alpha}
	\frac{ \left\{ \log(2+|\beta|)\right\}^{p_1 q_3}}
	{\langle \alpha+\beta \rangle^{(p-1)\mu/2-1}
	\jb{\alpha}^{p} 	\langle \beta \rangle^{\eta_1}} d\beta d\alpha.
\end{align}

First, suppose $r \geq t$.
Then, we get from \eqref{dai52} and \eqref{dai89}
\begin{align}\label{dai53}
	I(t,r) 
	&\lesssim \int_{|r-t|}^{t+r} 
	\frac{ \left\{ \log(2+\alpha)\right\}^{p_1 q_3}}{\langle \alpha \rangle^{1+\eta}} d \alpha \times
	\int_{r-t}^{t+r} \frac{1}{\langle \beta \rangle^{\eta_1}}
	d\beta.
\end{align}
Since $\eta>0$, we have from Lemma \ref{lem1} and  Lemma \ref{lem1.1}
\begin{align}
\int_{|r-t|}^{t+r} 
	\frac{ \left\{ \log(2+\alpha)\right\}^{p_1 q_3}}{\langle \alpha \rangle^{1+\eta}} d \alpha
	\lesssim& \frac{t \left\{ \log(2+|r-t|)\right\}^{q_3} }{\langle t+r \rangle \langle t-r \rangle^{\eta}}
	 \left\{ \log(2+t+r)\right\}^{(p_1-1)q_3}.\label{dai54}
\end{align}
On the other hand, to evaluate the $\beta$-integral, we shall use the following facts
about $\eta_1$ which can be deduced by \eqref{dai106}:
When $p<p_S(3+\mu)$, $\eta_1=1-\gamma_S(p,3+\mu) p_1/(2p)+(p-p_1)(p\nu-1)/p \ge 1-\gamma_S(p,3+\mu) p_1/(2p)$.
When $p=p_S(3+\mu)$, $p_1=1$ and $p\nu>1$, $\eta_1=1+(p-1)(p\nu-1)/p >1$.
And, when $p=p_S(3+\mu)$ and either $p_1=p$ or $p\nu=1$, $\eta_1=1$. 
These facts lead to
\begin{align} \nonumber
\int_{r-t}^{t+r} \frac{1}{\langle \beta \rangle^{\eta_1}}
	d\beta 
 &	\lesssim 
	\left\{
	\begin{array}{ll}
	\langle t+r \rangle^{\gamma_S(p,3+\mu)  p_1/(2p)} 
	& (p < p_S(3+\mu)),\\
	1 & (p=p_S(3+\mu),\ p_1=1 \ \mbox{and} \ p \nu >1),\\
	\log(2+t+r) & (p=p_S(3+\mu), \ \mbox{either} \ p_1=p \ \mbox{or} \ p \nu =1)
	\end{array}
	\right.
\\  \label{dai55}
 & = \langle t+r \rangle^{\gamma_S(p,3+\mu)  p_1/(2p)} \left\{ \log(2+t+r)\right\}^{q_4},
\end{align}
where we defined
$$
q_4=	\left\{
	\begin{array}{ll}
	0 & (\mbox{either}\ p < p_S(3+\mu)
	\ \mbox{or} \ p=p_S(3+\mu),\ p_1=1 \ \mbox{and} \ p \nu >1), \\
	1 & (p=p_S(3+\mu), \ \mbox{either} \ p_1=p \ \mbox{or} \ p \nu =1).
	\end{array}
	\right.
$$
Therefore, from \eqref{dai53}, \eqref{dai54} and \eqref{dai55}, we find
\begin{align}\label{dai55+}
	I(t,r) 
	\lesssim  \frac{t  \left\{ \log(2+|r-t|)\right\}^{q_3}
	}
	{\langle t-r \rangle^{\eta}} \times
	\frac{\left\{ \log(2+t+r)\right\}^{(p_1-1)q_3+q_4} }
	{\langle t+r \rangle^{1-\gamma_S(p,3+\mu)p_1/(2p)}}.
\end{align}
We note that if $r_1>0$, $r_2\geq0$, then $g(s)
=s^{-r_1} \left\{ \log (1+s) \right\}^{r_2}$ \ $(s \geq 1)$ is decreasing for large $s$.
Since $1-\gamma_S(p,3+\mu)p_1/(2p)
=p_1\eta+(p-p_1)/p
> \eta>0$,
by recalling \eqref{defD2}, we then get
\begin{align*}
		\frac{\langle t \rangle \left\{ \log(2+t+r)\right\}^{(p_1-1)q_3+q_4} }
	{\langle t+r \rangle^{1-\gamma_S(p,3+\mu)p_1/(2p)}} 
	\lesssim D_2(t)^{p_1/p}.
	\end{align*}
Moreover, for $r \ge t$, we have
$$
\frac{t}{\langle t \rangle} \lesssim \frac{r}{\langle t+r \rangle}.
$$
Indeed, if $r \geq 1$, then we have 
$r \sim \jb{t+r}$, so that the estimate holds.
On the other hand, if $0\leq r \leq 1$, then $\jb{t+r} \sim 1$,
thus we obtain the needed estimates. 
Now, 
from (\ref{dai55+}) 
we get (\ref{dai51}) for $r \geq t$.

Next, suppse $t \geq r$.
We get from (\ref{dai52}) and \eqref{dai89}
\begin{align*}
	I(t,r) 
	&\lesssim \int_{t-r}^{t+r} 
	\frac{ \left\{ \log(2+\alpha)\right\}^{p_1 q_3}}{\jb{\alpha}}
	\int_{-\alpha}^{\alpha}
	\frac{1}
	{\langle \alpha+\beta \rangle^{\eta}
	\langle \beta \rangle^{\eta_1}} d\beta d\alpha.
\end{align*}
Since $0<\eta <1-\gamma_S(p.3+\mu) p_1/(2p) \leq \eta_1$, it follows from Lemma \ref{lem3.1} with $k_1=\eta$, $k_2=0$, $k_3=1-\gamma_S(p,3+\mu)p_1/(2p)-\eta$ (resp. $k_3=\eta_1-\eta$) when  $p<p_S(3+\mu)$ 
(resp. $p=p_S(3+\mu)$) that
\begin{align*}
\int_{-\alpha}^{\alpha}\frac{1}{\jb{\alpha+\beta}^{\eta} \jb{\beta}^{\eta_1}} d\beta
\lesssim \jb{\alpha}^{-\eta}
\jb{\alpha}^{\gamma_S(p,3+\mu) p_1/(2p)}
\left\{ \log(2+ \alpha) \right\}^{q_4}.
\end{align*}
Therefore, we get
from \eqref{defD2}, Lemma \ref{lem1} and Lemma \ref{lem1.1}
\begin{align*}
	I(t,r) 
	&\lesssim 
	     \int_{t-r}^{t+r} \frac{\jb{\alpha}^{\gamma_S(p,3+\mu) p_1/(2p)}  \left\{  \log(2+ \alpha)\right\}^{p_1 q_3+q_4}}
	{\jb{\alpha}^{1+\eta}} d \alpha 
\\
	&\lesssim D_2(t+r)^{p_1/p}
	\int_{t-r}^{t+r} \frac{\left\{ \log(2+ \alpha)\right\}^{q_3}}
	{\jb{\alpha}^{1+\eta}} d \alpha\\
	&\lesssim
	\frac{ r \left\{ \log(2+|t-r|)\right\}^{q_3}}{\langle t+r \rangle
	\langle t-r \rangle^{\eta}} D_2(T)^{p_1/p}.
\end{align*}
Hence we obtain (\ref{dai51}).\\
(ii) \ $\eta \leq 0$.\\
Since $p\nu \geq 1$, we obtain from \eqref{dai109}
\begin{align*}
	 \langle r \rangle^{\mu/2} \langle t+r \rangle
	\langle t-r \rangle^{1/p}|u_L(t,r)| \leq r \| u_L \|_1.
\end{align*}
Similarly to the argument in the case of $\eta>0$, we obtain 
\begin{align*}
	|I_{-}(G)(t,r) |\lesssim \langle r \rangle^{-\mu/2}  
	\| u_L \|_1^{p-p_1} \| v \|_2^{p_1} I(t,r),
\end{align*}
where we put
\begin{align} \nonumber 
I(t,r)& := \iint_{\Delta_{-}(t,r)}
\frac{y 
\left\{
 \Psi(\sigma-y,\sigma+y)
	\right\}^{p_1 q_2} 
\left\{\log(2+\sigma+y)\right\}^{p_1 q_3}}
{\langle y \rangle^{(p-1)\mu/2}
\langle \sigma+y \rangle^{\eta_2}
\langle \sigma-y \rangle^{\eta_3}
} dy d\sigma
\\
\label{dai58}
       &\lesssim \int_{|r-t|}^{t+r} 
	\int_{r-t}^{\alpha}
	\frac{(\alpha+\beta) 
	\left\{
	\Psi(\beta, \alpha)
	\right\}^{p_1 q_2}
	\left\{ \log(2+\alpha) \right\}^{p_1 q_3}}
	{\langle \alpha+\beta \rangle^{(p-1)\mu/2}
	\langle \alpha \rangle^{\eta_2}
	\langle \beta \rangle^{\eta_3} } d\beta d\alpha.
\end{align}
We shall show
\begin{align}
	I(t,r)\lesssim& \frac{r\left\{
	\Psi(t-r,t+r)
	\right\}^{q_2} \left\{ \log (2+t+r) \right\}^{q_3}}{
	\langle t+r \rangle^{1+\eta}} D_2(T)^{p_1/p}
	.\label{dai57}
\end{align}

First, suppose $r/2 \geq t$.
Since $\eta>-1$ for $p>1$, we have $\eta_2=p_1\eta+p \ge 0$.
Besides, we have $\eta_3 \geq 0$.
Therefore, 
we get from (\ref{dai58}), \eqref{dai107},  and Lemma \ref{lem2}
\begin{align*}
	I(t,r) 
	&\lesssim \frac{r \left\{ \log(2+t+r) \right\}^{p_1 q_3}}
	{\langle r-t \rangle^{2+(1+p_1) \eta+\eta_3}} \int_{r-t}^{t+r} 
	\int_{r-t}^{\alpha}
	\left\{
	\Psi(\beta,\alpha)
	\right\}^{p_1 q_2}
	d\beta d\alpha\\
	&\lesssim \frac{r t^2 \left\{ \log(2+t+r) \right\}^{p_1 q_3}}
	{\langle r-t \rangle^{2+(1+p_1) \eta+\eta_3}} \\
        &\lesssim \frac{r\jb{t}^2 \left\{ \log(2+t+r) \right\}^{q_3}}{\langle t+r \rangle^{2+(1+p_1) \eta+\eta_3}}
         \left\{ \log(2+t+r) \right\}^{(p_1-1) q_3}\\
         &\lesssim \frac{r \left\{ \log(2+t+r) \right\}^{q_3}}{\langle t+r \rangle^{1+\eta}}
         \langle t \rangle^{1-p_1\eta-\eta_3}\left\{ \log(2+t) \right\}^{(p_1-1) q_3},
\end{align*}
because $(r+t)/3 \leq r-t$ for $r/2 \ge t$, and $1+p_1 \eta+\eta_3>0$
which follows from the equality $1+p\eta=(p-1)((\mu/2+1)p-1)$.
Then, from \eqref{dai106+}, we obtain \eqref{dai57}.

Next, suppose $t \geq r/2$.
We note that $\eta \leq 0$ leads to $(p-1)\mu/2 \leq  1$.
Then, we get from (\ref{dai58})
\begin{align*}
	I(t,r) 	
\lesssim \int_{|r-t|}^{t+r} 
	\int_{-\alpha}^{\alpha}
	\frac{ \langle \alpha \rangle^{1-(p-1)\mu/2}
	\left\{
	\Psi(\beta, \alpha)
	\right\}^{p_1 q_2}
	\left\{ \log(2+\alpha) \right\}^{p_1 q_3}}
	{\langle \alpha \rangle^{\eta_2}
	\langle \beta \rangle^{\eta_3} } d\beta d\alpha.
\end{align*}
Since $\left\{ \Psi(\beta,\alpha) \right\}^{p_1q_2}
/\jb{\beta}^{\eta_3}$ is an even function of $\beta$, we obtain 
\begin{align*}
	I(t,r) 
	\lesssim \left\{ \log(2+t+r) \right\}^{p_1 q_3} \int_{|r-t|}^{t+r} 
	\frac{1}{\langle \alpha \rangle^{1+(1+p_1) \eta}}
	\int_{0}^{\alpha}
	\frac{\left\{
	\Psi(\beta,\alpha)
	\right\}^{p_1q_2}}{\langle \beta \rangle^{\eta_3}}
	d\beta d\alpha,
\end{align*}
by \eqref{dai107}.
When $p_1=p$, we have $\eta_3=0$, so that  Lemma \ref{lem2} can be applied.
On the other hand, when $p_1=1$, we have $0<\eta_3<1$ and use the following inequality
\begin{align}
\int_{0}^{\alpha}
\log  \left( \frac{1+\alpha}{1+\beta} \right) (1+\beta)^{-\eta_3} d \beta 
&\lesssim(1+\alpha)^{1-\eta_3}, \quad \alpha \geq 0, \ 0<\eta_3<1,
\label{dai70+}
\end{align}
which is verified by the integration by parts.
Then we get from \eqref{dai96} and \eqref{dai70+}
\begin{align}
\int_{0}^{\alpha}
	\frac{
	\left\{ \Psi(\beta,\alpha) \right\}^{p_1 q_2}}{\jb{\beta}^{\eta_3}}
	d\beta \lesssim \jb{\alpha}^{1-\eta_3}.\label{dai70}
\end{align}
Hence, recalling \eqref{dai106+}, we get from \eqref{dai70} and Lemma \ref{lem1}
\begin{align*}
I(t,r) 
&\lesssim  \left\{ \log(2+t+r) \right\}^{p_1 q_3}  \int_{|r-t|}^{t+r} 
	\frac{1}{\langle \alpha \rangle^{1+(1+p_1) \eta-1+\eta_3}}d\alpha\\
&\lesssim \left\{ \log(2+t+r) \right\}^{p_1 q_3}
\langle t+r \rangle^{1-p_1 \eta-\eta_3}
\int_{|r-t|}^{t+r} 
	\frac{1}{\langle \alpha \rangle^{1+\eta}}d\alpha\\
	&\lesssim  \frac{r 
	\left\{ \Psi(t-r,t+r) \right\}^{q_2}
	\left\{ \log(2+t+r) \right\}^{p_1 q_3}
	}{\langle t+r \rangle^{1+\eta}} (1+t)^{\gamma_S(p,3+\mu)p_1/(2p)}\\
	&\lesssim  \frac{r \left\{ \Psi(t-r,t+r) \right\}^{q_2}\left\{ \log(2+t+r) \right\}^{ q_3}
	}{\langle t+r \rangle^{1+\eta}} D_2(T)^{p_1/p}.
\end{align*}
Therefore we obtain \eqref{dai57}. This completes the proof.
\end{proof}

\noindent
{\bf End of the proof of Theorem \ref{bw3} with $p \nu \geq 1$}.\,
Let $Y$ be the linear space defined by
\begin{align*}
Y=\left\{ v(t,r) \in C([0,T) \times (0,\infty)) \ ; \ \| v \|_{2} < \infty \right\}.
\end{align*}
We shall construct a local solution of integral equation \eqref{dai92} in the Banach space $(Y,\| \cdot \|_2)$. We define the sequence of functions $\{ v_n\}$ by
\begin{align*}
v_{1}=0,\quad  v_{n+1}=I_{-}(|\varepsilon u_{L} + v_n|^{p}/r^{p-1})
\quad (n=1,2,3,\cdots).
\end{align*}
We set
\begin{align*}
M_0&=2^{p-1} \widetilde{C}_1 \widetilde{C}_0^p,\\
\widetilde{C}_3&=(2^{2(p+1)}p)^p \max \{\widetilde{C}_2 M_0^{p-1}, (\widetilde{C}_2 \widetilde{C}_0^{p-1})^p \},
\end{align*}
where $\widetilde{C}_i$ \ $(0\leq i \leq 2)$ are positive constants given in Lemma \ref{fil1}, Lemma \ref{lem8} and Lemma \ref{lem9}. 
Then, analogously to the proof of Theorem 2.1 in \cite{KTW19} and Theorem 2.2 in \cite{IKTW20}, we see that $\{ v_n \}$ is a Cauchy sequence in $Y$ provided that the 
inequality
\begin{align}
\widetilde{C}_3 \varepsilon^{p(p-1)}D_2(T) \leq 1\label{dai93}
\end{align}
holds, where $D_2(T )$ is defined in \eqref{defD2}. Since $Y$ is complete, there exists a function $v \in Y$ such that $v_n$ converges to $v$ in $Y$. Therefore $v$ satisfies \eqref{dai92}.

Note that \eqref{llifespan} follows  from (\ref{dai93}). We shall show this fact only in the case of $1<p<p_{S}(3+\mu)$ and 
$p\nu=1$, since the other cases can be proved similarly. By definition of 
$b(\varepsilon)$ in \eqref{bdef}, we know that $b(\varepsilon)$ is decreasing in $\varepsilon$ and 
$\displaystyle \lim_{\varepsilon \to 0+0} b(\varepsilon)=\infty$.
Let us fix $\varepsilon_0 > 0$ as
\begin{align}
1<\widetilde{C}_4 b(\varepsilon_0),\label{dai94}
\end{align}
where $\widetilde{C}_4:=\min\{2^{-1}, (2 \widetilde{C}_3^{2/\gamma_S(p,3+\mu)})^{-1}\}$. For $0<\varepsilon \leq \varepsilon_0$,
we take $T$ to satisfy
\begin{align}
1 \leq T <\widetilde{C}_4  b(\varepsilon), \label{dai95}
\end{align}
so that $1+T \le 2T \le 2\widetilde{C}_4  b(\varepsilon)$.
Then, since $q_3=1$ by $p\nu=1$, $\widetilde{C}_3 (2\widetilde{C}_4)^{\gamma_S(p,3+\mu)/2} \leq 1$, 
and $2\widetilde{C}_4 \leq 1$, it follows that 
\begin{align*}
\widetilde{C}_3 \varepsilon^{p(p-1)}D_2(T )
&\leq \widetilde{C}_3 \varepsilon^{p(p-1)}(2T)^{\gamma_S(p,3+\mu)/2} \left\{ \log(1+2T) \right\}^{p-1}\\
&\leq \widetilde{C}_3 (2\widetilde{C}_4)^{\gamma_S(p,3+\mu)/2}\varepsilon^{p(p-1)}
b(\varepsilon)^{\gamma_S(p,3+\mu)/2} \left\{ \log(1+2\widetilde{C}_4 
b(\varepsilon) ) \right\}^{p-1}\\
&\leq \varepsilon^{p(p-1)}
b(\varepsilon)^{\gamma_S(p,3+\mu)/2}  \left\{ \log(1+b(\varepsilon) ) \right\}^{p-1}\\
& = 1,
\end{align*}
by \eqref{bdef}.
Hence, if we assume \eqref{dai94} and \eqref{dai95}, then \eqref{dai93} holds. Therefore \eqref{llifespan} in the case is obtained for $0<\varepsilon \leq \varepsilon_0$. This completes the proof.

\section{Upper bounds of the lifespan}
Let $u$ denote the solution of the problem \eqref{eq.io0}
in what follows.
When $\varphi\equiv 0$ and $\psi \geq 0$,
it follows from \eqref{eq.io0}, \eqref{eq.io00} and \eqref{eq.le2} that
\begin{align} \label{eq.lb2}
& u(t,r) \gtrsim \varepsilon u_L(t,r)
+\widetilde{I_-}(|u|^p/y^{p-1})(t,r),
\\  \label{eq.lb20}
& u_L(t,r) \gtrsim \widetilde{J_-}(\psi)(t,r)
\end{align}
hold for $t$, $r>0$, where we put
\begin{align}\label{eq.lb2a}
& \widetilde{I_-}(F)(t,r) =  \iint_{\Delta_-(t,r)} 
\frac{\langle -t+\sigma +r+ y \rangle^{\mu}}{\langle r \rangle^{\mu/2} \langle y \rangle^{\mu/2}}
F(\sigma, y) dy d\sigma,
\\
\label{eq.lb2b}
& \widetilde{J_-}(\psi)(t,r) = \int_{ |t-r| }^{t+r} 
\frac{\langle r-t + y \rangle^{\mu}}{\langle r \rangle^{\mu/2} \langle y \rangle^{\mu/2}}
\psi(y) dy.
\end{align}

Our first step is to obtain basic lower bounds of 
$u_L$ defined by \eqref{eq.io00}.

\begin{lm}\label{Lemma3}
Assume that \eqref{hyp1} holds.
Then there exists $M=M(\mu,\kappa)>0$ such that
\begin{align*}
u_L(t,r)\ge \frac{M}{r^{\mu/2}(t-r)^{\nu} }
\quad\mbox{for}\ t\ge r+1  \ \mbox{and} \ t\le 2r.
\end{align*}
Here $\nu$ is defined by \eqref{dai98}, i.e., $\nu=\kappa-(\mu/2)-1$.
\end{lm}

\begin{proof}
Let $t\ge r+1$ and $t\le2r$.
By \eqref{hyp1}, \eqref{eq.lb20} and \eqref{eq.lb2b}
we have
\begin{align} \notag
u_L(t,r)\gtrsim
& \int_{t-r}^{t+r}
\frac{\langle r-t+y\rangle^\mu}{ \langle r\rangle^{\mu/2} \langle y 
\rangle^{\mu/2}}
\frac{1}{(1+y)^\kappa}
dy
\\ \notag
\gtrsim & 
\frac{1}{\langle r\rangle^{\mu/2}}
\int_{t-r}^{t+r}
\frac{(y-t+r)^\mu}{y^{\kappa+(\mu/2)}}
dy,
\end{align}
because $y \ge t-r \ge 1$.
Applying Lemma \ref{KO}, we get
\begin{align} \notag
u_L(t,r) \gtrsim \,
\frac{1}{\langle r\rangle^{\mu/2}(t-r)^{\nu}}
\left(1-\frac{t-r}{t+r}\right)^{\mu+1},
\end{align}
which implies the desired estimate, because $r \ge 1$ and 
$(t-r)/(t+r)\le1/3$ for  $t\ge r+1$, $t\le2r$.
This completes the proof.
\end{proof}

Our next step is to derive iterative lower bounds of 
the solution to \eqref{eq.io0}.

\begin{lm}\label{Lemma6}
For $T_0>1$ we set
$$
\Sigma:=\{(t, r) \in[0,\infty)\times [0,\infty); \ t\ge r+T_0, \, t\le 2r\}.
$$
Let $u(t,r)$ be the continuous solution of  \eqref{eq.io0}
and $a$, $b\ge 0$, $M_1>0$. 
If $u_L(t,r) \geq 0$ and 
\begin{align} \label{hyp2}
u(t,r)\ge \frac{M_1(t-r-T_0)^a}{r^{\mu/2}(t-r)^b}
\quad \mbox{for}\ (t,r)\in \Sigma,
\end{align}
then there exists $C>0$, which is independent of
$a$, $b$ and $M_1$, such that
\begin{align*}
u(t,r)\ge \frac{CM_1^p (t-r-T_0)^{pa+2} }{(pa+2)^2 r^{\mu/2}(t-r)^{pb+(p-1)(\mu/2 +1)} }
\quad \mbox{for}\ (t,r)\in \Sigma.
\end{align*}
\end{lm}

\begin{proof}
For $(t,r)\in \Sigma$, we put
$$
Q:=\{(\sigma,y) \in [0,\infty) \times[0,\infty); \ t-r\le y, \, \sigma+y
\le 3(t-r), \, T_0\le\sigma-y\le t-r\}.
$$
In the following, let $(t,r)\in \Sigma$.
Since $Q\subset\Delta_-(t,r)$ and $Q \subset \Sigma$, we have from \eqref{eq.lb2}, \eqref{eq.lb2a} and \eqref{hyp2}
\begin{align} \notag
u(t,r)\gtrsim
& \iint_{Q}
\frac{\langle -t+\sigma+r+y\rangle^\mu}{ \langle r\rangle^{\mu/2} \langle y \rangle^{\mu/2}}
\frac{|u(\sigma,y)|^p}{y^{p-1}}
dyd\sigma
\\ \notag
\gtrsim & 
\frac{M_1^p}{\langle r\rangle^{\mu/2}}
\iint_{Q}
\frac{\langle -t+\sigma+r+y\rangle^\mu
(\sigma-y-T_0)^{pa}}
{\langle y \rangle^{(p+1)\mu/2+p-1} (\sigma-y)^{pb}}
dyd\sigma.
\end{align}
Changing the variavles by $\alpha=\sigma+y$, $\beta=\sigma-y$, since $y \ge t-r \ge T_0>1$
for $(\sigma,y)\in Q$, we get
\begin{align} \notag
u(t,r)\gtrsim & \,
\frac{M_1^p}{\langle r\rangle^{\mu/2}}
\int_{T_0}^{t-r}d\beta
\int_{2(t-r)+\beta}^{3(t-r)}
\frac{(\alpha-t+r)^\mu(\beta-T_0)^{pa}}
{(\alpha-\beta)^{(p+1)\mu/2+p-1} \beta^{pb}}d\alpha
\\ \notag
\gtrsim & \,
\frac{M_1^p}{\langle r\rangle^{\mu/2} (t-r)^{pb+(p-1)(\mu/2+1)}}
\int_{T_0}^{t-r}
(\beta-T_0)^{pa} (t-r-\beta)
d\beta
\\ \notag
= & \,
\frac{M_1^p (t-r-T_0)^{pa+2}}{ (pa+1)(pa+2)
\langle r\rangle^{\mu/2} (t-r)^{pb+(p-1)(\mu/2+1)}},
\end{align}
by the integration by parts.
Thus we get the desired estimate,
because $r \ge T_0>1$.
This completes the proof.
\end{proof}

\noindent
{\bf End of the proof of Theorem \ref{bw1}}.\,
We divide the argument into three cases.\\
(i) \ $1<p<p_F(\kappa)$.

From \eqref{eq.lb2} and Lemma \ref{Lemma3}
\begin{align*}
	u(t,r)\geq \frac{M \varepsilon (t-r-T_0)^{\mu/2+1} }{r^{\mu/2} (t-r)^{\kappa}}
	\quad \mbox{for} \ (t,r) \in \Sigma.
\end{align*}
Therefore, by Lemma \ref{Lemma6}, we obtain
\begin{align} \label{LB1}
u(t,r)\ge
\frac{C_n (t-r-T_0)^{a_n} }{r^{\mu/2}(t-r)^{b_n} }
\quad \mbox{for}\ (t,r)\in \Sigma, \ n=0,1,\cdots,
\end{align}
where we put
\begin{align*}
& a_{n+1}=pa_n+2,\ a_0=\mu/2+1, \quad
\\
& b_{n+1}=pb_n+(\mu/2+1)(p-1),\ b_0=\kappa,
\\
& C_{n+1}\ge \frac{C\, C_n^p}{(pa_n+2)^2},\ C_0=
M\varepsilon,
\end{align*}
where $C$ is the number from Lemma \ref{Lemma6}.
It is easy to see that
$$
a_n=\left( \frac{2}{p-1}+\frac{\mu}{2}+1 \right)p^n-\frac{2}{p-1}, \quad b_n=(\kappa+\mu/2+1)
p^n-(\mu/2)-1,
$$
and hence there exists a constant $0<D<1$ such that 
$$
C_{n+1}\ge \frac{C\, C_n^p}{(a_{n+1})^{2}}
\ge \frac{DC_n^p}{p^{2n}}.
$$
Then, we get
\begin{align} \notag
\log C_n
& \ge \log D\sum_{j=0}^{n-1}p^j
-(2\log p)\sum_{j=1}^{n-1}jp^{n-j-1}+p^n
\log C_0
\\ \notag
& \ge p^n \left\{ \frac{\log D}{p-1} -(2\log p) \sum_{j=1}^{\infty}j p^{-j-1}+
\log (M\varepsilon) \right\}
\\  \notag
& \ge p^n \log (\varepsilon E) 
\end{align}
with a suitable positive constant $E$. 
Therefore, from \eqref{LB1} we get
\begin{align} \label{LB2}
u(t,r)\ge \frac{(t-r-T_0)^{-2/(p-1)}(t-r)^{\mu/2+1}}
{r^{\mu/2}}\exp (p^n \log J(t,r))
\end{align}
for $ (t,r) \in \Sigma$, where we set
\begin{align*} 
J(t,r)=\varepsilon E(t-r-T_0)^{2/(p-1)+\mu/2+1}(t-r)^{-\kappa-(\mu/2)-1}.
\end{align*}
Let $(t,r)=(\tau,\tau/2)$ and $\tau\ge 4T_0$.
Then we get $t-r-T_0\ge \tau/4$ and $(t,r)\in\Sigma$, so that
$J(\tau,\tau/2) \ge \varepsilon E_1 \tau^{\gamma_F(p,\kappa)/(p-1) }$,
where $\gamma_F(p,\kappa)=2-(p-1)\kappa$ and we put $E_1=2^{\kappa-4/(p-1)-(\mu/2)-1}E$.
Now we specify
$$
\tau= (2^{-1}E_1 \varepsilon)^{-(p-1)/\gamma_F(p,\kappa)}.
$$
Since $\gamma_F(p,\kappa)>0$ for $p<p_F(\kappa)$, we may assume
$\tau\ge 4T_0$, by choosing $\varepsilon$ suitably small.
Thus we get $J(\tau,\tau/2)\ge 2$.
Then, we see from \eqref{LB2} that 
$u(\tau,\tau/2) \to \infty$ as $n \to \infty$.
This means that the lifespan $T(\varepsilon)\le \tau$, which yields the desired estimate.\\
(ii) \ $1<p<p_3(3+\mu)$ and $p \nu=1$.

Since $p \nu =1$, we get from \eqref{eq.lb2} and Lemma \ref{Lemma3}
\begin{align}
	u(t,r)\geq \frac{M \varepsilon}{r^{\mu/2} (t-r)^{1/p}}
	\quad \mbox{for} \ r+1 \le t \le 2r.
\label{dai114}
\end{align}
Then, similarly to the argument in the proof of Lemma \ref{Lemma6}, we obtain
\begin{align}
u(t,r) &\gtrsim \frac{\varepsilon^{p}}{\jb{r}^{\mu/2}} \int_{1}^{t-r} \int_{2(t-r)+\beta}^{3(t-r)} \frac{(\alpha-t+r)^{\mu}}{(\alpha-\beta)^{(p+1)\mu/2+p-1} \beta} d \alpha d \beta\nonumber\\
&\gtrsim \frac{\varepsilon^p}{r^{\mu/2} (t-r)^{(\mu/2+1)(p-1)}}
\int_{1}^{t-r} \frac{t-r-\beta}{\beta} d \beta
\quad \mbox{for} \ r+1 \le t \le 2r.
\label{dai110}
\end{align}
Let $(t,r) \in \Sigma$ and $T_0 > 2$ in the following. 
Then, we have
\begin{align}
\int_{1}^{t-r} \frac{t-r-\beta}{\beta} d \beta
&=\int_{1}^{t-r} \log \beta d \beta 
\geq \int_{T_0/2}^{t-r} \log \beta d \beta\nonumber\\
&\geq (t-r-T_0/2) \log(T_0/2)
 \ge  \frac{t-r}{2} \log(T_0/2).  
 \label{dai111}
\end{align}
Hence, we get from \eqref{dai110} and \eqref{dai111}
\begin{align*}
u(t,r) \ge \frac{\widetilde{M} \log(T_0/2)\varepsilon^p (t-r-T_0)}{r^{\mu/2}(t-r)^{(\mu/2+1)(p-1)}}\quad \mbox{for} \ (t,r) \in \Sigma,
\end{align*}
where $\widetilde{M}$ is a suitable positive constant.
Therefore, by Lemma \ref{Lemma6}, we obtain \eqref{LB1} with
\begin{align*}
a_{0}=1, \ b_{0}=(\mu/2+1)(p-1) \ \mbox{and} \
C_0=\widetilde{M} \log (T_0/2)  \varepsilon^p,
\end{align*}
Then, similar computation as in the previous case leads to
\begin{align*}
	u(t,r) \geq \frac{(t-r-T_0)^{-2/(p-1)}}{r^{\mu/2}}
	\exp(p^n  \log \widetilde{J}(t,r)), \ (t,r)\in\Sigma,
\end{align*}
where we set
\begin{align*}
	\widetilde{J}(t,r)=\widetilde{E} \varepsilon^p \log (T_0/2)
	(t-r-T_0)^{(p+1)/(p-1)} (t-r)^{-(\mu/2+1)p} 
\end{align*}
with a suitable positive constant $\widetilde{E}$.

Let $(t,r)=(\tau,\tau/2)$ and $\tau=4T_0$.
Then we get $t-r-T_0=T_0$ and $(t,r)\in\Sigma$.
Therefore, we find a constant $\widetilde{E}_1 \in (0,1)$, which is independent of
$\varepsilon$ and $T_0$, such that
\begin{align}
	\widetilde{J}(\tau, \tau/2) \ge 2 \widetilde{E}_{1}\varepsilon^p
	(T_0/4)^{\gamma_S(p,3+\mu)/(2(p-1))} \log (T_0/2),\label{dai113}
\end{align}
where $\gamma_S(p,n)$ is defined by \eqref{gSdef}.
Now we define $T_0=T_0(\varepsilon)$ by
\begin{align*}
	T_0=4 \widetilde{E}_{1}^{-2(p-1)/\gamma_S(p,3+\mu)} b(\varepsilon).
\end{align*}
Recalling the fact that 
$b(\varepsilon)$ is monotonically decreasing in $\varepsilon$ and 
	$\displaystyle \lim_{\varepsilon \to 0+0} b(\varepsilon)=\infty$,
we may assume that $\widetilde{E}_{1}^{-2(p-1)/\gamma_S(p,3+\mu)} b(\varepsilon) \geq 1$,
by choosing $\varepsilon$ sufficiently small.
Then, we see that $T_0 > 2$ and 
\begin{align*}
	{T_0}/{2}=2\widetilde{E}_{1}^{-2(p-1)/\gamma_S(p,3+\mu)} b(\varepsilon)
	\geq 1+b(\varepsilon),
\end{align*}
because $\gamma_S(p,3+\mu)>0$ when $1<p<p_{S}(3+\mu)$. 
Therefore, we get from \eqref{dai113} and \eqref{bdef}
\begin{align*}
\widetilde{J}(\tau,\tau/2)
 \geq 2 \varepsilon^p b(\varepsilon)^{\gamma_S(p,3+\mu)/(2(p-1))}
\log(1+b(\varepsilon))
=2.
\end{align*}
Hence, as in the previous case, we obtain the desired estimate.\\
(iii) \ $p=p_{S}(3+\mu)$ and $p\nu=1$.\\  
For $\rho>0$, we define the following quantity:
\begin{align*}
\langle u \rangle(\rho)
=\inf\{ \jb{y}^{\mu/2} (\sigma-y)^{\eta}|u(\sigma,y)|; \ 0 \le \sigma \le 2y, \ \sigma-y \ge \rho\}.
\end{align*}
Since $p \nu=1$ and $\nu = \eta$ by $p=p_{S}(3+\mu)$,
it follows from \eqref{dai114} that
\begin{align*}
	\jb{u}(y)\ge C_1 \varepsilon \quad (y \geq1).
\end{align*}
As is shown in Section 4 in \cite{KK22}, it holds that
\begin{equation*}
\langle u \rangle(y) \ge
C_2 \int_{1}^{y} \left(1-\frac{\xi}{y}\right)
\frac{[\langle u \rangle(\xi)]^p}
{\xi^{p \eta}}\, d\xi, \ \ y \ge 1.
\end{equation*}
Therefore, using Lemma \ref{lem:boi-1} with $\alpha=1$, $\beta=0$, $\theta=p \eta =1$, we get
\begin{align*}
	T_{*}(\varepsilon) \leq \exp(C \varepsilon^{-(p-1)}).
\end{align*}
Since $T(\varepsilon) \leq T_{*}(\varepsilon)$, we obtain the desired estimate.
Summing up the conclusions in (i), (ii) and (iii), we finish
the proof of Theorem \ref{bw1}.

\vspace{3mm}


\end{document}